\documentclass{amsart}

\usepackage{amssymb}
\usepackage{mathrsfs}
\usepackage{hyperref}
\usepackage{enumerate}

\newtheorem{theorem}{Theorem}
\newtheorem{lemma}[theorem]{Lemma}

\theoremstyle{definition}
\newtheorem{definition}[theorem]{Definition}
\newtheorem{example}[theorem]{Example}

\theoremstyle{remark}

\newcommand\fix{{\rm{Fix}}}
\newcommand\orbit{{\rm{Orb}}}
\newcommand\divides{{\mathchoice{\mathrel{\bigm|}}{\mathrel{\bigm|}}{\mathrel{|}}%
       {\mathrel{|}}}}
\newcommand\smalldivides{\mathrel{\kern-2pt\kern3.5pt|}}
\newcommand\notdivides{\mathrel{\kern-3pt\not\!\kern4.3pt\bigm|}}
\newcommand\smallnotdivides{\mathrel{\kern-2pt\not\!\kern3.5pt|}}

\begin{document}

\title{Time-changes preserving zeta functions}


\author{Sawian Jaidee}
\address{Department of Mathematics, Faculty of Science, Khon Kaen University, Thailand.}
\email{jsawia@kku.ac.th}
\thanks{The first author thanks Khon Kaen University
for funding a research visit
to Leeds University, and thanks the Research Visitors'
Centre at Leeds for their hospitality.
The third author thanks the Lorentz Centre, Leiden,
where much of the work was done.}

\author{Patrick Moss}
\email{pbsmoss2@btinternet.com}
\thanks{}

\author{Tom Ward}
\address{Ziff Building 13.01, University of Leeds, LS2 9JT, U.K.}
\email{t.b.ward@leeds.ac.uk}
\thanks{}

\subjclass[2010]{Primary 37P35, 37C30, 11N32}

\date{\today}

\dedicatory{To Graham Everest (1957--2010), in memoriam}

\commby{}

\begin{abstract}
We associate to any
dynamical system with finitely many
periodic orbits of each period a
collection of possible time-changes
of the sequence of periodic point counts
for that specific system
that preserve the property of counting periodic points
for some system.
Intersecting over all dynamical systems
gives a monoid of time-changes that have
this property for all such systems.
We show that the only polynomials
lying in this monoid are
the monomials, and that this monoid
is uncountable. Examples give some
insight into how the structure of the
collection of maps
varies for different dynamical systems.
\end{abstract}

\maketitle


\section{Introduction}\label{sectionIntroduction}

We are concerned with operations (that we will call
time-changes) that act on integer
sequences and preserve the following property.
An integer sequence~$(a_n)$
is called \emph{realizable} if there is a map~$T\colon X\to X$
with the property that
\[
a_n=\fix_{(X,T)}(n)=\vert\{x\in X\mid T^nx=x\}\vert
\]
for all~$n\geqslant1$. In this case we will also
say that the sequence~$(a_n)$ is \emph{realized} by
the \emph{system}~$(X,T)$.
If we require~$X$ to be a compact metric
space and~$T$
to be a homeomorphism, or indeed if we require~$T$ to be
a~$C^{\infty}$ diffeomorphism of the~$2$-torus,
then the same collection of sequences is
characterized by this definition
(by work of Puri and the last author~\cite{MR1873399}
or Windsor~\cite{MR2422026}, respectively).
Notice that not all integer sequences are
realizable: certainly if~$(a_n)$ is realizable then~$a_n\geqslant 0$
for all~$n\geqslant1$, but there are congruence conditions as well.
For
example,~$a_2-a_1$ is the number of points that live on
closed orbits of length precisely~$2$ under the map~$T$, so~$a_2-a_1$
must be both non-negative and even.

Certain operations on integer sequences preserve the
property of being realizable
for trivial reasons. If~$(a_n)$ is realized by~$(X,T)$
and~$(b_n)$ by~$(Y,S)$, then the product sequence~$(a_nb_n)$ is realized
by the Cartesian product~$T\times S\colon X\times Y\to X\times Y$,
defined by~$(T\times S)(x,y)=(T(x),S(y))$ for all~$(x,y)\in X\times Y$.
Similarly, the
sum~$(a_n+b_n)$ is realized by the disjoint union~$T\sqcup S\colon X\sqcup Y\to X\sqcup Y$,
where~$T\sqcup S$ is defined as
\[
(T\sqcup S)(z)=\begin{cases}T(z)&\mbox{if }z\in X;\\
S(z)&\mbox{if }z\in Y.\end{cases}
\]

All these statements may also be expressed in terms of the
\emph{dynamical zeta function}
of~$(X,T)$, formally
defined as~$\zeta_{(X,T)}(z)
=
\exp\left(\sum_{n\geqslant1}\fix_{(X,T)}(n)\frac{z^n}{n}\right)$.
Here we are interested in properties of
the collection of all possible dynamical zeta functions.
Thus, for example, the space of all zeta functions is closed
under multiplication, because the sum of two realizable sequences
is realizable, and is closed under a Hadamard-like
formal multiplication because the product is.
We refer to work of Carnevale and Voll~\cite{MR3808651}
or Pakapongpun and the last author~\cite{MR2486259,MR3194906}
for more on the
combinatorial and analytic properties of these
`functorial' operations on realizable sequences.

A different kind of operation on sequences (or on zeta
functions) is a \emph{time change}, defined as follows.
Any function~$h\colon\mathbb{N}\to\mathbb{N}$
defines an operation on integer sequences
by sending~$(a_n)$ to~$(a_{h(n)})$.
If the original sequence~$(a_n)$ is realized by~$(X,T)$,
then this may be thought of as replacing the
sequence of iterates~$T,T^2,T^3,\dots$, whose fixed
point counts give the sequence~$(a_n)$, with
the time-changed sequence~$T^{h(1)},T^{h(2)},T^{h(3)},\dots$.
The question we are interested in is this: counting the number of
points fixed by those iterates~$T^{h(1)},T^{h(2)},T^{h(3)},\dots$
gives an integer sequence. Is it possible that this sequence
counts periodic points for some (other) system~$(Y,S)$?

\begin{definition}\label{definition}
For a map~$T\colon X\to X$ with~$\fix_{(X,T)}(n)<\infty$ for all~$n\geqslant1$,
define
\[
\mathscr{P}(X,T)=\{h\colon\mathbb{N}\to\mathbb{N}\mid\bigl(\fix_{(X,T)}(h(n))\bigr)
\mbox{ is a realizable sequence}\}
\]
to be the set of \emph{realizability-preserving time-changes for}~$(X,T)$.
Also define
\[
\mathscr{P}=\bigcap_{\{(X,T)\}}\mathscr{P}(X,T)
\]
to be the monoid of \emph{universally
realizability-preserving time-changes}, where
the intersection is taken over all systems~$(X,T)$
for which~$\fix_{(X,T)}(n)<\infty$ for all~$n\geqslant1$.
\end{definition}

Some remarks are in order.
\begin{enumerate}[(a)]
\item Clearly the identity map defined by~$h(n)=n$
for all~$n\in\mathbb{N}$
lies in~$\mathscr{P}(X,T)$
for any system~$(X,T)$. Thus~$\mathscr{P}$ is non-empty.
\item If functions~$h_1,h_2$
lie in~$\mathscr{P}$,
then their composition~$h_1\circ h_2$
also lies in~$\mathscr{P}$, because
by definition if~$(a_n)$ is a realizable sequence
then~$(a_{h_2(n)})$ is also realizable, and so~$(a_{h_1(h_2(n))})$
is too. Thus~$\mathscr{P}$ is a
monoid inside the monoid of
all maps~$\mathbb{N}\to\mathbb{N}$ under composition.
\item Notice that~$\mathscr{P}(X,T)$ is a certain
collection of functions
defined by~$(X,T)$, but it will
typically be some \emph{other system}~$(Y,S)$
that bears
witness to the statement~$h\in\mathscr{P}(X,T)$,
by satisfying the property
\[
\fix_{(Y,S)}(n)=\fix_{(X,T)}(h(n))
\]
for all~$n\geqslant1$.
\item The requirement that~$\fix_{(X,T)}(n)<\infty$ for all~$n\geqslant1$
is natural for the type of question we are interested in, and
will be assumed of all systems from now on.
\end{enumerate}

It is not obvious that any non-trivial maps~$h$ could
have either of the properties in Definition~\ref{definition},
but the following simple examples show
how functions with this type of property can arise.

\begin{example}\label{exampleTrivialSystem}
If~$\vert X\vert=1$, then~$\fix_{(X,T)}(n)=1$ for all~$n\geqslant1$,
so the sequence of periodic point counts for~$(X,T)$ is
the constant sequence~$(1,1,1,\dots)$. Any function~$h\colon\mathbb{N}\to\mathbb{N}$ time-changes
this constant sequence to itself, so lies in~$\mathscr{P}(X,T)$
because it is realized by the system~$(X,T)$ itself.
Thus in this case~$\mathscr{P}(X,T)=\mathbb{N}^{\mathbb{N}}$
is the monoid of all maps~$\mathbb{N}\to\mathbb{N}$.
\end{example}

\begin{example}\label{examplehconstant}
If~$h\colon\mathbb{N}\to\mathbb{N}$ is a constant function,
with~$h(n)=k$ for all~$n\geqslant1$, then for any system~$(X,T)$
the time-change produces the constant sequence whose~$n$th
term is~$\fix_{(X,T)}(k)$ for all~$n\geqslant1$.
This sequence is realized by the system~$(Y,S)$, where~$\vert Y\vert=\fix_{(X,T)}(k)$
and~$S$ is the identity map. That is,~$h\in\mathscr{P}$.
\end{example}

\begin{example}\label{exampleMakes2032Impossible}
For any system~$(X,T)$ we clearly have~$\fix_{(X,T)}(2n)=\fix_{(X,T^2)}(n)$
for all~$n\geqslant1$, because the~$2n$th iterate of~$T$ is
the~$n$th iterate of~$T^2$. Thus the map~$h$ defined by~$h(n)=2n$
for all~$n\geqslant1$ is a member of~$\mathscr{P}(X,T)$ for any
system~$(X,T)$, and so is universally realizability-preserving.
Thus~$h\in\mathscr{P}$.
\end{example}

Our purpose is to prove two results about the structure of~$\mathscr{P}$,
and describe some examples that expose more subtle
possibilities for the collection of functions~$\mathscr{P}(X,T)$.

\begin{theorem}\label{theoremA}
A polynomial lies in~$\mathscr{P}$ if and only if it is a monomial.
\end{theorem}

We illustrate what is going on in Theorem~\ref{theoremA}
with examples.
Some of these involve
statements about
specific dynamical systems, and an adequate
reference for these results is~\cite[Ch.~11]{MR1990179}.

\begin{example}
(a) Let~$(X,T)$ denote the `golden mean'
system. This is one of a family of
maps called shifts of finite type. It is defined on the
space
\[
X=\{(x_n)_{n\in\mathbb{Z}}\in\{0,1\}^{\mathbb{Z}}\mid x_k=1\implies x_{k+1}=0\mbox{ for all }k\in\mathbb{Z}\}
\]
by
the left shift, so~$T$ sends~$(x_n)_{n\in\mathbb{Z}}$ to the
sequence whose~$k$th term is~$x_{k+1}$ for all~$k\in\mathbb{Z}$.
Then it may be shown that
\begin{equation}\label{equationAllWalksOfLife}
\fix_{(X,T)}(n)={\rm{trace}}\begin{pmatrix}1&1\\1&0\end{pmatrix}^n
\end{equation}
for all~$n\geqslant1$,
so~$\fix_{(X,T)}(n)$ is the~$n$th Lucas
number, the sequence of periodic point counts
begins~$(1,3,4,7,11,\dots)$, and~$\zeta_{(X,T)}(z)=\frac{1}{1-z-z^2}$.
The Cartesian square~$T\times T$
is also a shift of finite type,
and a calculation shows that
\[
\zeta_{(X\times X,T\times T)}(z)=\tfrac{1}{(1+z)(1-2z-2z^2+z^3)}.
\]
Theorem~\ref{theoremA} asserts in part that
the map~$h$ defined by~$h(n)=n^2$ for all~$n\geqslant1$
lies in~$\mathscr{P}$.
In particular, this means that there
is \emph{some} system~$(Y,S)$ whose sequence of
periodic point counts is obtained by sampling
the Lucas sequence along the squares,
namely~$(1,7,76,2207,\dots)$.
Such a system cannot
be conjugate to a shift of finite type,
because~$\limsup_{n\to\infty}\frac{1}{n^2}\log\fix_{(Y,S)}(n)
=\log\bigl(\frac{1+\sqrt{5}}{2}\bigr)>0$, while shifts of finite
type have periodic point counts that only grow exponentially fast,
because they can be expressed in terms of the trace
of powers of an integer matrix as in~\eqref{equationAllWalksOfLife}.

\noindent(b) In the reverse direction, Theorem~\ref{theoremA} says
that the map~$h$ defined by~$h(n)=n^2+1$ for all~$n\geqslant1$
is not universally realizability-preserving.
This means there
must be \emph{some} system~$(X,T)$ with the property that
time-changing by sampling its periodic point
counts along the
polynomial~$n^2+1$ produces an integer sequence
which cannot be the periodic point count of \emph{any} map.
A system that bears witness to the fact that~$h\notin\mathscr{P}$
may be constructed as follows.
Let~$X=\mathbb{N}$, and define a map~$T\colon X\to X$
as follows:
\begin{itemize}
\item $T(1)=1$, so the subset~$\{1\}$ consists of a single closed
orbit of length~$1$ for~$T$;
\item $T(2)=3$, and~$T(3)=2$, so the subset~$\{2,3\}$ consists
of a single closed orbit of length~$2$ for~$T$;
\item $T(4)=5,T(5)=6$, and~$T(6)=4$, so the subset~$\{4,5,6\}$ consists
of a single closed orbit of length~$3$ for~$T$;
\end{itemize}
and so on, resulting in a system~$(X,T)$ which has exactly
one closed orbit of length~$n$ for every~$n\geqslant1$.
We will write~$\orbit_{(X,T)}(n)=1$ for all~$n\geqslant1$ to express this.
Now~$\fix_{(X,T)}(n)=\sum_{d\smalldivides n}d\orbit_{(X,T)}(d)=\sigma(n)$ (the sum
of divisors of~$n$), since the points fixed by~$T^n$
are exactly the union of the~$d$ points lying on each closed orbit of length~$d$
for each divisor~$d$ of~$n$. Thus the sequence of periodic
point counts for~$(X,T)$ begins~$(1,3,4,7,6,12,\dots)$.
Time-changing this along the polynomial
given by~$n^2+1$ gives the
sequence~$(3,6,18,18,42,\dots)$ which cannot
count the periodic points of any map, as
such a map would need to have~$\frac{6-3}{2}$ closed
orbits of length~$2$.

\noindent(c) A Lehmer--Pierce sequence, with~$n$th
term~$\vert\det(A^n-I)\vert$ for some
integer matrix~$A$, counts periodic points
for an ergodic toral endomorphism if it is
non-zero for all~$n\geqslant1$. Time-changing it
along the squares then gives a sequence that
counts periodic points for some map, and this
sequence has a characteristic quadratic-exponential
growth rate, resembling a `bilinear' or `elliptic'
divisibility sequence. However, it will have fundamentally
different arithmetic properties, and
cannot be an elliptic sequence by work of
Luca and the last author~\cite{MR3576291}.
\end{example}

Theorem~\ref{theoremA} suggests that~$\mathscr{P}$ is
(unsurprisingly) small,
but work of the second author may be used to
show that there are many other
maps in~$\mathscr{P}$, resulting in the following
result.
This will be proved in Section~\ref{includesproofofTheorem7}.

\begin{theorem}\label{theoremB}
The monoid~$\mathscr{P}$ is uncountable.
\end{theorem}

\section{Proofs of Theorem~\ref{theoremA}}

First we recall from~\cite{MR1873399}
that an integer sequence~$(a_n)$ is
realizable if and only if
\begin{equation}\label{equationbasiccongruence}
\frac{1}{n}\sum_{d\smalldivides n}\mu({n}/{d})a_d
=
\frac{1}{n}\sum_{d\smalldivides n}\mu(d)a_{n/d}
\in\mathbb{N}_0
\end{equation}
for all~$n\geqslant1$,
where~$\mu$ denotes the M{\"o}bius function.
Equivalently,~$(a_n)$ is realizable if and only
if~$(\mu*a)(n)$ is non-negative and divisible by~$n$
for all~$n\geqslant1$, where~$*$ denotes Dirichlet convolution.

The condition~\eqref{equationbasiccongruence}
characterizes realizability because we have
\[
a_n=\fix_{(X,T)}(n)
=
\sum_{d\smalldivides n}d\orbit_{(X,T)}(d)
\]
for all~$n\geqslant1$ if and only
if
\[
\orbit_{(X,T)}(n)
=
\frac{1}{n}\sum_{d\smalldivides n}\mu({n}/{d})\fix_{(X,T)}(d)
=
\frac1n(\mu*a)(n)
\]
is the number of closed
orbits of length~$n$ under~$T$, for all~$n\geqslant1$.

\begin{proof}[Proof of `if' in Theorem~\ref{theoremA}: monomials preserve realizability.]
We follow the method of the thesis~\cite{pm}
of the second author.
Assume that~$h(n)=cn^k$
for some~$c\in\mathbb{N}$
and~$k\in\mathbb{N}_0$.

If~$k=0$, then the result is clear, as the
constant
sequence~$(a_{c},a_{c},a_{c},\dots)$
is realized by the space comprising~$a_c$
points all fixed by a map (as mentioned in Example~\ref{examplehconstant} above).
If~$(a_n)$ is realized by~$(X,T)$,
then~$(a_{cn})$ is realized by~$(X,T^c)$ for any~$c\in\mathbb{N}$
(as mentioned in Example~\ref{exampleMakes2032Impossible} above for~$c=2$),
so it is enough to consider the case~$h(n)=n^k$
for some~$k\geqslant1$.

Assume therefore that~$(a_n)$ is realizable --- which
for this argument we think of as satisfying~\eqref{equationbasiccongruence}
rather than in terms of a system that realizes the sequence ---
and write~$b_{n^{\vphantom{k}}}=a_{n^k}$ for~$n\geqslant1$.
We wish to show property~\eqref{equationbasiccongruence}
for the sequence~$(b_n)$.
Fix~$n\in\mathbb{N}$, and let~$n=p_1^{n_1}\cdots p_r^{n_r}$
be its prime decomposition, with~$n_j\geqslant1$
for~$j=1,\dots,r$.
Then
\begin{equation}\label{equationpm1}
(\mu*b)(n)
=
a_{n^k}
-\sum_{p_i}a_{n^k/p_i^k}
+\sum_{p_i,p_j}a_{n^k/p_i^kp_j^k}
-\cdots+
(-1)^ra_{n^k/p_1^k\cdots p_r^{k}}
\end{equation}
where~$p_i,p_j,\dots$ are distinct
members of~$\{p_1,\dots,p_r\}$.
Let
\[
\delta=n^k/p_1^{k-1}\cdots p_r^{k-1},
\]
so in particular~$n\divides\delta$.
Let
\begin{equation}\label{TMS3}
e=\sum_{\substack{m\smalldivides n^k\\\delta\smalldivides m}}
\sum_{d\smalldivides m}\mu({m}/{d})a_d.
\end{equation}
Since~$(a_n)$ is realizable, we have
by~\eqref{equationbasiccongruence}
that
\[
m\divides\sum_{d\smalldivides m}\mu({m}/{d})a_d\geqslant0,
\]
so in particular~$e\geqslant0$ and~$n\divides e$.
Thus it is enough to show that~$e=(\mu*b)(n)$.
Let~$m\divides n^k$ with~$\delta\divides m$, so
that we may write
\begin{equation}\label{TMS7}
m=p_1^{k(n_1-1)+j_1}\cdots p_r^{k(n_r-1)+j_r}
\end{equation}
with~$1\leqslant j_1,\dots,j_r\leqslant k$.
Thus by~\eqref{TMS3} we have
\[
e=\sum_{j_1=1}^{k}\cdots\sum_{j_r=1}^{k}
\sum_{d\smalldivides m}\mu(d)a_{m/d}
\]
with~$m$ given by~\eqref{TMS7}.
Let
\begin{equation}\label{equationTMS13}
m_1=m/p_1^{k(n_1-1)+j_1}=p_2^{k(n_2-1)+j_2}\cdots
p_r^{k(n_r-1)+j_r}.
\end{equation}
Then we have
\[
\sum_{d\smalldivides m}\mu(d)a_{m/d}=
\sum_{d\smalldivides m_1}
\mu(d)\bigl(a_{m/d}-a_{m/p_1d}\bigr).
\]
Thus, because~$m_1$ is independent of~$j_1$,
\[
\sum_{j_1=1}^{k}\sum_{d\smalldivides m}
\mu(d)a_{m/d}=\sum_{d\smalldivides m_1}\sum_{j_1=1}^{k}
\mu(d)\bigl(a_{m/d}-a_{m/p_1d}\bigr)
\]
and hence
\[
\sum_{j_1=1}^{k}\sum_{d\smalldivides m}
\mu(d)a_{m/d}
=
\sum_{d\smalldivides m_1}
\mu(d)\bigl(
a_{p_1^{kn_1}m_1/d}-a_{p_1^{k(n_1-1)}m_1/d}
\bigr).
\]
It follows from~\eqref{TMS3} that
\[
e=\sum_{j_2=1}^{k}\cdots\sum_{j_r=1}^{k}
\sum_{d\smalldivides m_1}
\mu(d)
\bigl(
a_{p_1^{kn_1}m_1/d}-a_{p_1^{kn_1}m_1/p_1^kd}
\bigr),
\]
where~$m_1$ is given by~\eqref{equationTMS13}.
The same procedure may be repeated,
first setting
\[
m_2=m_1/p_2^{k(n_2-1)+j_2},
\]
to obtain~$e=e_1-e_2$, where
\[
e_1=\sum_{j_3=1}^{k}\cdots\sum_{j_r=1}^{k}
\sum_{d\smalldivides m_2}
\mu(d)
\bigl(
a_{p_1^{kn_1}p_2^{kn_2}m_2/d}-a_{p_1^{kn_1}p_2^{kn_2}m_2/p_2^{k}d}
\bigr)
\]
and
\[
e_2=\sum_{j_3=1}^{k}\cdots\sum_{j_r=1}^{k}
\sum_{d\smalldivides m_2}
\mu(d)\bigl(
a_{p_1^{kn_1}p_2^{kn_2}m_2/p_1^{k}d}
-
a_{p_1^{kn_1}p_2^{kn_2}m_2/p_1^{k}p_2^kd}
\bigr).
\]
Continuing inductively shows that each expression
obtained matches up with a term in~\eqref{equationpm1},
as required.
\end{proof}

\begin{proof}[Proof of `only if' in Theorem~\ref{theoremA}: only monomials preserve realizability.]
This argument proceeds rather differently, because
we are free to construct dynamical systems with convenient
properties to constrain what the polynomial can be.
So assume that
\[
h(n)=c_k+c_{k-1}n+c_{k-2}n^2+\cdots+c_0n^k
\]
is a polynomial in~$\mathscr{P}$ with~$c_0\neq0$,~$k\geqslant1$, and~$h(\mathbb{N})\subset\mathbb{N}$.
For completeness we recall the following well-known
result.

\begin{lemma}\label{lemmaEuclid}
The coefficients of~$h$ are
rational, and the set of primes
dividing some~$h(n)$ with~$n\in\mathbb{N}$
is infinite.
\end{lemma}

\begin{proof}
We have
\[
\begin{pmatrix}
h(1)\\
h(2)\\
h(3)\\
\vdots\\
h(k+1)
\end{pmatrix}
=
\begin{pmatrix}
1&1&1&\cdots&1\\
1&2&4&\cdots&2^k\\
1&3&9&\cdots&3^k\\
\vdots\\
1&(k+1)&(k+1)^2&\cdots&(k+1)^k
\end{pmatrix}
\begin{pmatrix}
c_k\\
c_{k-1}\\
c_{k-2}\\
\vdots\\
c_0
\end{pmatrix},
\]
and the determinant~$\prod_{1\leqslant i<j\leqslant k+1}(j-i)$
of this matrix (a so-called `Vandermonde' determinant,
an instance of Stigler's law~\cite{MR3155603})
is non-zero, so the coefficients of~$h$ are all rational.

Turning to the prime divisors of the
values of~$h$, if~$c_k=0$ the claim
is clear, and if~$k=1$ then~$c_0$ and~$c_1$ are
integers so we may
write~$c_1+c_0n$ as~$\gcd(c_1,c_0)\bigl(c_1'+c_0'n\bigr)$
with~$\gcd(c_1',c_0')=1$ to see this,
so assume that~$c_k\neq0$ and~$k>1$.
Then we may write~$h(n)=np(n)+c_k$ for some polynomial~$p$ of
positive degree. We may not have~$p(\mathbb{N})\subset\mathbb{N}$ of
course, but~$h$ (and hence~$p$) certainly has
rational coefficients.
Then we have
\[
\frac{m!c_k^2p(m!c_k^2)+c_k}{c_k}
=
m!c_kp(m!c_k^2)+1
=
\frac{h(m!c_{k}^2)}{c_k}.
\]
If~$m$ is large then~$p(m!c_k^2)$ is an integer because~$p$
has rational coefficients and~$c_k$ is rational, so~$h(m!c_k^2)$
must be divisible by some prime greater than~$m$.
\end{proof}

Using Lemma~\ref{lemmaEuclid}, we let~$q$ be a very large
prime dividing some value of~$h$, let~$n_0$ be
the smallest value of~$n$ such that~$q\divides h(n)$, and
let~$(X,T)$ consist of a single orbit of
length~$q$.
(Looking further ahead, it is here that
we are failing to solve question~(e)
from Section~\ref{sectionQuestions}, in that
we choose the system using information from the
candidate polynomial rather than universally.)
Then, by construction,
\begin{equation}\label{equationdefinesa}
a_n=\fix_{(X,T)}(n)=\begin{cases}
0&\mbox{if }q\notdivides n;\\
q&\mbox{if }q\divides n.
\end{cases}
\end{equation}
Thus the assumption that~$h\in\mathscr{P}$ means that~$(a_{h(n)})$
is a realizable sequence, and we know from~\eqref{equationdefinesa}
that it only takes
on the values~$0$ and~$q$.
Since~$q$ is prime, we have
\[
a_{h(1)}\equiv a_{h(q)}\pmod{q}
\]
by~\eqref{equationbasiccongruence}.
Since~$(a_{h(n)})$ only takes the
values~$0$ and~$q$,
we deduce from~\eqref{equationdefinesa} that~$n_0$ is
the smallest~$n$ such that~$a_{h(n)}=q$.
Thus the sequence~$(a_{h(n)})$
starts
\begin{equation}\label{equationTMS1}
(a_{h(n)})=(0,\dots,0,q,\dots)
\end{equation}
with the first~$q$ in the~$h(n_0)$th place.
Now~$(a_{h(n)})$ is, by the assumption that~$h\in\mathscr{P}$, realizable
by some dynamical system~$(Y,S)$,
so~\eqref{equationTMS1}
says that~$S$ has no fixed points, no points of period~$2$, and so on,
but it has~$q$ points of period~$h(n_0)$.
By~\eqref{equationbasiccongruence} this is only
possible if~$h(n_0)\divides q$, so we deduce
that
\begin{equation}\label{equationsmallmiracle}
h(n_0)=q.
\end{equation}
Now consider the points of period~$2n_0$ in~$(Y,S)$.
There are~$a_{h(2n_0)}$ of these points,
and of course any point fixed by~$S^{n_0}$
is also fixed by~$S^{2n_0}$,
so
\[
a_{h(2n_0)}\geqslant a_{h(n_0)}=q.
\]
On the other hand, the sequence~$(a_{h(n)})$
only takes on the values~$0$ and~$q$,
so in fact
\[
a_{h(2n_0)}=q.
\]
The same argument shows that~$a_{h(jn_0)}=q$
for all~$j\geqslant1$.
By~\eqref{equationdefinesa},
it follows that~$q\divides h(jn_0)$ for all~$j\geqslant1$.
Thus we have
\begin{align*}
h(n_0)=c_k+c_{k-1}n_0+\cdots+c_0n_0^k&\equiv0,\\
h(2n_0)=c_k+c_{k-1}2n_0+\cdots+c_02^kn_0^k&\equiv0,\\
&\hspace{6pt}\vdots\\
h((k+1)n_0)=c_k+c_{k-1}(k+1)n_0+\cdots+c_0(k+1)^kn_0^k&\equiv0\\
\end{align*}
modulo~$q$.
That is,
\[
\begin{pmatrix}
1&1&1&\cdots&1\\
1&2&4&\cdots&2^k\\
1&3&9&\cdots&3^k\\
\vdots\\
1&(k+1)&(k+1)^2&\cdots&(k+1)^k
\end{pmatrix}
\begin{pmatrix}
c_k\\
c_{k-1}n_0\\
c_{k-2}n_0^2\\
\vdots\\
c_0n_0^{k}
\end{pmatrix}
\equiv
\begin{pmatrix}
0\\
0\\
0\\
\vdots\\
0
\end{pmatrix}
\]
modulo~$q$.
Since~$k$ is fixed and~$q$ is large,
the determinant~$\prod_{1\leqslant i<j\leqslant k+1}(j-i)$
of this matrix is non-zero
modulo~$q$, so we deduce that the matrix is invertible
modulo~$q$, and hence
\begin{equation}\label{equationcoefficientsconclusion}
c_{k-j}n_0^j\equiv0\pmod{q}
\end{equation}
for~$j=0,\dots,k$.


Now, by definition,~$n_0$
is the smallest~$n$ with~$q\divides h(n)$,
which tells us nothing about the size of~$n_0$.
However, we have seen in~\eqref{equationsmallmiracle}
that the realizability preserving property
shows that~$h(n_0)=q$.
It follows that for large~$q$
we have
\[
n_0\approx\Bigl(\frac{q}{c_0}\Bigr)^{1/k}\ll q
\]
since~$c_0\neq0$.
So~\eqref{equationcoefficientsconclusion} shows that
\[
c_{k-j}n_0^j\approx c_{k-j}\Bigl(\frac{q}{c_0}\Bigr)^{j/k}\ll q
\]
for~$j\leqslant k-1$,
and therefore the congruences in~\eqref{equationcoefficientsconclusion}
in fact imply a list of equalities,
\[
c_k=c_{k-1}=\cdots=c_1=0
\]
because we can choose~$q$ to be as large as we please.
It follows that~$h(n)=c_0n^k$, as claimed.
We can of course deduce nothing about~$c_0$,
because~$c_0n_0^k\approx q$.
\end{proof}

\section{Examples and Proof of Theorem~\ref{theoremB}}
\label{includesproofofTheorem7}

The statement that monomials
are realizability-preserving in Theorem~\ref{theoremA}
may be applied in several ways to give (potentially)
new results about existing sequences as follows.
If~$(a_n)$ is an integer sequence known to be realized
by some system~$(X,T)$, then
Theorem~\ref{theoremA} says that~$(a_{n^k})$ is also
realizable for any~$k\in\mathbb{N}$.
The basic relation~\eqref{equationbasiccongruence}
then allows us to deduce three types
of result:
\begin{itemize}
\item \emph{Congruences} in the spirit of Fermat's little theorem,
because
\[
\sum_{d\smalldivides n}\mu(n/d)a_{d^k}\equiv0
\]
modulo~$n$ for
all~$n\geqslant1$.
\item \emph{Positivity statements},
because~$\sum_{d\smalldivides n}\mu(n/d)a_{d^k}\geqslant0$ for
all~$n\geqslant1$.
\item \emph{Integrality statements}, because
the collection of all closed
orbits for a system~$(X,T)$ may be
thought of as a disjoint union of individual orbits,
showing that
\[
\zeta_{(X,T)}(z)=
\exp\left(
\sum_{n\geqslant1}\fix_{(X,T)}(n)\tfrac{z^n}{n}
\right)
=
\prod_{n\geqslant1}\bigl(1-z^n\bigr)^{-\orbit_{(X,T)}(n)},
\]
so the Taylor expansion of~$\zeta_{(X,T)}(z)$ at~$z=0$
automatically
has integer coefficients, and hence
the Taylor expansion of~$\exp\left(
\sum_{n\geqslant1}a_{n^k}\tfrac{z^n}{n}
\right)$ at~$z=0$ has integral coefficients.
\end{itemize}

The congruence statements may be thought of as generalizations
of Fermat's little theorem because of the following
simple example.

\begin{example}
The full shift~$T$ on~$a\geqslant2$ symbols (that is, the left
shift on the sequence space~$X=\{1,2,\dots,a\}^{\mathbb{Z}}$)
has~$\fix_{(X,T)}(n)=a^n$ for all~$n\geqslant1$.
Following the three observations above, we deduce from
Theorem~\ref{theoremA} the following
statements, for any~$k\in\mathbb{N}$
and for all~$n\geqslant1$:
\begin{itemize}
\item $\sum_{d\smalldivides n}\mu(n/d)a^{d^k}\equiv0$ modulo~$n$, so
in particular we have~$a^{p^k}\equiv a$ modulo~$p$ for any prime~$p$;
\item $\sum_{d\smalldivides n}\mu(n/d)a^{d^{k}}\geqslant 0$;
\item the Taylor expansion of~$\exp\left(
\sum_{n\geqslant1}a^{n^k}\tfrac{z^n}{n}
\right)$ at~$z=0$ has integer coefficients.
\end{itemize}
\end{example}

These statements are all straightforward, but the
same conclusions hold starting from any realizable sequence~$(a_n)$.
To illustrate the type of conclusions one may reach,
we list some less straightforward examples.
Links to the \href{https://oeis.org/}{Online Encyclopedia
of Integer Sequences}~\cite{OEIS} are included for
convenience.
In each case a family of congruence, positivity, and
integrality results of the same shape follow
from Theorem~\ref{theoremA}.

\begin{itemize}
\item The Bernoulli numerators~$(\tau_n)$
or denominators~$(\beta_n)$,
define by~$\left\vert\frac{B_{2n}}{2n}\right\vert=\frac{\tau_n}{\beta_n}$
in lowest terms for all~$n\geqslant1$, where~$
\frac{t}{{\rm{e}}^t-1}
=\sum_{n=0}^{\infty}B_n\frac{t^n}{n!}$
(see~\href{http://oeis.org/A027641}{A27641},
shown to be realizable in~\cite{MR1938222};
\href{http://oeis.org/A002445}{A2445}
shown to be realizable in~\cite{pm}, respectively).
\item The Euler numbers~$\bigl((-1)^nE_{2n}\bigr)$,
where~$\frac{2}{{\rm{e}}^t+{\rm{e}}^{-t}}=\sum_{n=0}^{\infty}E_n\frac{t^n}{n!}$
(see~\href{http://oeis.org/A000364}{A364}, shown to
be realizable in~\cite{pm}).
\item The Lucas sequence~$(1,3,4,7,11,\dots)$
(see~\href{https://oeis.org/A000204}{A204}
and~\cite{MR1866354} for its special status
as a realizable sequence).
\item The divisor sequence~$(\sigma(n))=(1,3,4,7,6,12,8,\dots)$.
\end{itemize}

\begin{example}\label{integrality}
The following sequences of coefficients
are integral, answering questions raised in
the relevant \href{https://oeis.org/}{Online Encyclopedia
of Integer Sequences} entry.
\begin{itemize}
\item The sequence \href{https://oeis.org/A166168}{A166168}
    is the sequence of Taylor coefficients of the zeta
    function of the dynamical system with
    periodic point data given by time-changing
    the Lucas sequence along the squares,
    and so is integral as conjectured
    in the Online Encyclopedia
of Integer Sequences.
    More generally, the same property holds for
    the Lucas sequence sampled along any integer power.
\item We have~$\exp\left(\sum_{n\geqslant1}\sigma(n)\tfrac{z^n}{n}\right)
=\sum_{n\geqslant0}p(n)z^n$,
where~$p$ is the partition function \href{https://oeis.org/A000041}{A41};
time-changing along the squares gives
as Taylor coefficients the
Euler transform of the Dedekind~$\psi$ function.
The argument here shows that sampling along
any power also gives integral Taylor coefficients.
\end{itemize}
\end{example}

Because of the diversity of integer sequences
satisfying the condition~\eqref{equationbasiccongruence},
it is clear that the property of preserving realizability
is extremely onerous. Indeed, the forward
direction of Theorem~\ref{theoremA} (stating that
monomials are universally realizability-preserving) is a little surprising,
and one might ask if there are any further functions
with this property. In fact Moss~\cite{pm}
has constructed many such maps.

\begin{lemma}\label{lemmagpconstruction}
Let~$p$ be a prime, and define~$g_p\colon\mathbb{N}\to\mathbb{N}$ by
\[
g_p(n)=\begin{cases}n&\mbox{if }p\notdivides n;\\pn&\mbox{if }p\divides n.\end{cases}
\]
Then~$g_p$ lies in~$\mathscr{P}$.
\end{lemma}

\begin{proof}
Let~$(a_n)$ be a realizable sequence
and write~$(b_n)=(a_{g_p(n)})$. We need to show that~$(b_n)$
satisfies~\eqref{equationbasiccongruence}.
Fix~$n$, and write~$n=p^{\text{ord}_p(n)}m$ with~$\gcd(m,p)=1$.

Assume first that~$\text{ord}_p(n)=0$. Then~$p\notdivides n$ and so
\[
\sum_{d\smalldivides n}\mu({n}/{d})b_d
=
\sum_{d\smalldivides n}\mu({n}/{d})a_d
\]
and so~$(b_n)$ satisfies~\eqref{equationbasiccongruence} at~$n$.

Next assume that~$\text{ord}_p(n)=1$, so that~$n=pm$ and~$p\notdivides m$.
Then
\begin{align}
(\mu*b)(n)=\sum_{d\smalldivides pm}\mu(d)b_{pm/d}
&=
\sum_{d\smalldivides m}\mu(d)b_{n/d}+\mu(p)\sum_{d\smalldivides m}\mu(d)b_{m/d}\notag\\
&=
\sum_{d\smalldivides m}\mu(d)a_{p^2m/d}
-
\sum_{d\smalldivides m}\mu(d)a_{m/d}\label{fishybanana0}
\end{align}
since~$\mu$ is multiplicative.
Now
\begin{align}
(\mu*a)(pn)
=
(\mu*a)(p^2m)
&=
\sum_{d\smalldivides p^2m}\mu(d)a_{p^2m/d}\notag\\
&=
\sum_{d\smalldivides m}\mu(d)a_{p^2m/d}
-
\sum_{d\smalldivides m}\mu(d)a_{pm/d}\label{fishybanana1}
\end{align}
and
\begin{align}
(\mu*a)(n)
=
(\mu*a)(pm)
&=
\sum_{d\smalldivides pm}\mu(d)a_{pm/d}\notag\\
&=
\sum_{d\smalldivides m}\mu(d)a_{pm/d}
-
\sum_{d\smalldivides m}\mu(d)a_{m/d}\label{fishybanana2}.
\end{align}
Adding~\eqref{fishybanana1} and~\eqref{fishybanana2}
gives
\[
(\mu*a)(pn)+(\mu*a)(n)
=
\sum_{d\smalldivides m}\mu(d)a_{p^2m/d}-\sum_{d\smalldivides m}\mu(d)a_{m/d}
=(\mu*b)(n)
\]
by~\eqref{fishybanana0},
so~$(b_n)$ satisfies~\eqref{equationbasiccongruence}
at~$n$.

Finally, assume that~$\text{ord}_p(n)\geqslant2$.
Then
\[
\sum_{d\smalldivides n}\mu({n}/{d})b_d
=
\underbrace{\sum_{d\smalldivides m}\mu({n}/{d})a_d}_{\Sigma_0}
+
\sum_{j=1}^{\text{ord}_p(n)}
\underbrace{\sum_{d\smalldivides m}
\mu({n}/{p^jd})a_{pd}}_{\Sigma_j}
\]
Now~$\mu\bigl(\tfrac{n}{d}\bigr)=0$ for all~$d$ dividing~$m$, so~$\Sigma_0=0$.

Similarly,~$\mu\bigl(\tfrac{n}{p^jd}\bigr)=0$ for~$j\leqslant\text{ord}_p(n)-2$,
so~$\Sigma_j=0$ for~$1\leqslant j\leqslant\text{ord}_p(n)-2$.

For the two remaining terms, we have
\begin{align*}
\Sigma_{\text{ord}_p(n)}+\Sigma_{\text{ord}_p(n)-1}
&=
\sum_{d\smalldivides m}\mu({m}/{d})a_{pd}
+
\sum_{d\smalldivides m}\mu({pm}/{d})a_{pd}\\
&=
\sum_{d\smalldivides m}\mu({m}/{d})a_{pd}
-
\sum_{d\smalldivides m}\mu({m}/{d})a_{pd}
=0,
\end{align*}
so~\eqref{equationbasiccongruence} holds trivially for~$(b_n)$ at~$n$.

We deduce that~$(b_n)$
satisfies~\eqref{equationbasiccongruence}
for all~$n\geqslant1$, as required.
\end{proof}

\begin{proof}[Proof of Theorem~\ref{theoremB}.]
Let~$S=\{p_1,p_2,\dots\}\subseteq\{2,3,5,7,11,\dots\}$ be any set of
primes, and define~$g_S\colon\mathbb{N}\to\mathbb{N}$ formally
by~$g_S=g_{p_1}\circ g_{p_2}\circ\cdots$
in the notation of Lemma~\ref{lemmagpconstruction}.
For definiteness, we write a set of primes as~$\{p_{j_1},p_{j_2},\dots\}$
with~$p_{j_1}<p_{j_2}<\cdots$.
More precisely, the map~$g_S$
then may be defined
as follows. For~$n\in\mathbb{N}$ the
set
\[
\{p_j\mid p_j\mbox{ divides }n\}=\{p_{j_1},\dots,p_{j_t}\}
\]
is finite, and then we define
\[
g_S(n)=g_{p_{j_1}}\circ\cdots\circ g_{p_{j_{t}}}(n).
\]
If~$S$ and~$T$ are different subsets of the primes,
then there is a prime~$p$ in the symmetric difference
of~$S$ and~$T$, and clearly~$g_S(p)\neq g_T(p)$.
It follows that there are uncountably many
different functions~$g_S$.

Formally, we also need to slightly improve the simple
observation that~$\mathscr{P}$ is
a monoid in Section~\ref{sectionIntroduction}
(remark~(b) after Definition~\ref{definition}),
as follows. If~$(h_1,h_2,\dots)$ is a sequence
of functions in~$\mathscr{P}$ with
the property that
\[
\{j\in\mathbb{N}\mid h_j(n)\neq n\}
=
\{j_n^{(1)},j_n^{(2)},\dots,j_n^{(r_n)}\}
\]
is finite for any~$n\in\mathbb{N}$, then
the infinite composition~$h=h_1\circ h_2\circ\cdots$
defined by
\[
h(n)
=
h_{j_n^{(1)}}\circ\cdots\circ h_{j_n^{(r_n)}}(n)
\]
for any~$n\in\mathbb{N}$
is also in~$\mathscr{P}$.
This is clear, because for any
given~$n$ checking~\eqref{equationbasiccongruence}
only involves evaluating~$h$ on finitely many terms.
We deduce that there are uncountably many
different elements of~$\mathscr{P}$ from
Lemma~\ref{lemmagpconstruction}.
\end{proof}

\section{Dynamical systems with additional polynomial time-changes}

As mentioned in Example~\ref{exampleTrivialSystem},
if~$X$ simply comprises a single fixed point
for~$T$ then~$\mathscr{P}(X,T)=\mathbb{N}^{\mathbb{N}}$.
Less trivial systems will have fewer maps that
preserve realizability,
and the complex way in which properties of
a map relate to the structure of its associated
set of maps are illustrated here by
examples of systems~$(X,T)$ with
\begin{equation}\label{equationTyingDownRoof}
\mathscr{P}\subsetneq\mathscr{P}(X,T)\subsetneq\mathbb{N}^{\mathbb{N}}.
\end{equation}

\begin{example}\label{exampleLogReally}
Let~$T\colon X\to X$ be the full shift on~$a\geqslant2$ symbols,
so that we have~$\fix_{(X,T)}(n)=a^n$ for all~$n\geqslant1$.
Then we claim (this is an observation from the thesis
of the second named author~\cite{pm}) that
if~$h(n)=c_0+c_1n+\cdots+c_kn^k$ is any
polynomial with non-negative integer coefficients,
then~$h\in\mathscr{P}(X,T)$.
By Theorem~\ref{theoremA}, we know that
the sequence~$(a^{n^j})$ is realized by some
map~$T_j\colon X\to X$ for any~$j=1,\dots,k$. Certainly the
constant sequence~$(a,a,\dots)$ is
realized by the identity map~$T_0$ on a
set with~$a$ elements. Then the Cartesian
product
\[
S=
\underbrace{T_0\times\cdots\times T_0}_{\text{$c_0$ copies}}
\times\underbrace{T_1\times\cdots\times T_1}_{\text{$c_1$ copies}}
\times\cdots\times
\underbrace{T_k\times\cdots\times T_k}_{\text{$c_k$ copies}}
\]
acting on~$Y=X^{c_0+c_1+\cdots+c_k}$
has
\[
\fix_{(Y,S)}(n)
=
a^{c_0}\bigl(a^{n}\bigr)^{c_1}\cdots\bigl(a^{n^k}\bigr)^{c_k}
=a^{h(n)}
\]
for~$n\geqslant1$, by construction.
Thus~$h\in\mathscr{P}(X,T)$, showing that~$\mathscr{P}(X,T)$
is strictly larger than~$\mathscr{P}$.
On the other hand, if the map that
exchanges~$1$ and~$2$ (and fixes all
other elements of~$\mathbb{N}$) lies in~$\mathscr{P}(X,T)$,
then we must be able to find some dynamical
system~$S\colon Y\to Y$ with~$\fix_{(Y,S)}(1)=a^2$
and~$\fix_{(Y,S)}(2)=a$. This forces~$a^2\leqslant a$
(because every fixed point of a map is also fixed by the second iterate
of the map),
so~$a\leqslant1$. It follows that~$\mathscr{P}(X,T)$
is strictly smaller than~$\mathbb{N}^{\mathbb{N}}$,
since~$a\geqslant2$.
\end{example}

In general it is not at all easy to describe~$\mathscr{P}(X,T)$ --- indeed
with the exception of the trivial case~$\mathbb{N}^{\mathbb{N}}$
which arises for the identity map on a finite set,
we have no examples with a complete description any
more insightful than the definition.
Example~\ref{exampleLogReally} relies on the
accidental
fact that~$a^na^m=a^{n+m}$, allowing us to translate
Cartesian products of systems into addition in the
time-change. The next example
of a system satisfying~\eqref{equationTyingDownRoof}
relies on a different arithmetic trick, as well
as the result from Example~\ref{exampleLogReally}.

\begin{example}
Let~$T\colon X\to X$ be the map~$x\mapsto-ax$
modulo~$1$ on the additive circle~$X=\mathbb{R}/\mathbb{Z}$
for some integer~$a\geqslant2$.
Then we have~$\fix_{(X,T)}(n)=a^n-(-1)^{n}$ for~$n\geqslant1$,
and we claim that if~$h(n)=n^2+1$, then~$h\in\mathscr{P}(X,T)$.
(In fact, the same argument shows the same
property for any polynomial with non-negative coefficients,
but for simplicity of notation we consider this
specific example.)
To prove this, we first show that
\begin{equation}\label{fishbanana}
\eta(n)=\sum_{d\smalldivides n}(-1)^d\mu(n/d)
=0
\end{equation}
for all~$n>2$.
Writing~$\boldsymbol{\mu}(s)=\sum_{n\geqslant1}\frac{\mu(n)}{n^s}$,~$\boldsymbol{\zeta}$
for the Riemann zeta function,
and~$\boldsymbol{\eta}$ for the Dirichlet~$\eta$-function~$\sum_{n\geqslant1}\frac{(-1)^{n-1}}{n^s}$,
it is clear that~$\boldsymbol{\eta}(s)=(1-2^{-s})\boldsymbol{\zeta}(s)$
by splitting into odd and even terms,
and~$\boldsymbol{\zeta}\boldsymbol{\mu}=1$,
so~$\boldsymbol{\mu}(s)\boldsymbol{\eta}(s)=(1-2^{1-s})$
for~$\Re(s)>1$.
It follows that~$\eta(1)=-1$,~$\eta(2)=2$,
and~$\eta(n)=0$ for~$n>2$.

As all our other arguments are elementary,
for completeness we also show~\eqref{fishbanana} directly,
by separating out the power of~$2$ dividing~$n$,
as follows.
\begin{itemize}
\item If~$n>2$ is odd, then
\[
\eta(n)=
-\sum_{d\smalldivides n}\mu(n/d)=
-\sum_{d\smalldivides n}\mu(d)=0.
\]
\item If~$n=2^k$ for some~$k>1$, then
\[
\eta(n)
=
\sum_{d\smalldivides2^k}(-1)^d\mu(2^k/d)
=\mu(1)+\mu(2)=0.
\]
\item If~$n=2m$ with~$m>2$ odd, then
\begin{align*}
\eta(n)
=
\sum_{d\smalldivides2m}\mu(d)(-1)^{2m/d}
=
\sum_{d\smalldivides m}\mu(d)\bigl((-1)^{2m/d}
-
(-1)^{m/d}\bigr)
=
2\sum_{d\smalldivides m}\mu(d)=0.
\end{align*}
\item Finally, if~$n=2^km$ with~$k,m>1$ and~$m$
odd, then
\begin{align*}
\eta(n)
=
\sum_{d\smalldivides2^km}\mu(d)(-1)^{2^km/d}
=
\sum_{d\smalldivides m}\mu(d)
\bigl((-1)^{2^km/d}
-
(-1)^{2^{k-1}m/d}\bigr)=0.
\end{align*}
\end{itemize}
We now show that~$h\in\mathscr{P}(X,T)$ using
the basic relation~\eqref{equationbasiccongruence}.
That is, we need to show the
congruence and positivity properties
in~\eqref{equationbasiccongruence}
for the sequence~$(a_n)$
defined by~$a_n=a^{n^2+1}+(-1)^n$
for~$n\geqslant1$
(since~$(-1)^{n^2+1}=-(-1)^n$).
Now~$(a*\mu)(1)=a^2-1$ and~$(a*\mu)(2)=a^2(a^3-1)+2$,
so we see that~$(a*\mu)(n)$ is non-negative
and divisible by~$n$ for~$n=1,2$ as desired.
For~$n>2$, we have
\begin{equation}\label{equationNotUnsympathetic}
(a*\mu)(n)=\sum_{d\smalldivides n}\mu(n/d)a^{d^2+1}
+\sum_{d\smalldivides n}(-1)^d\mu(n/d)
=\sum_{d\smalldivides n}\mu(n/d)a^{d^2+1}
\end{equation}
since~$\eta(n)=0$. Now a
special case of Example~\ref{exampleLogReally}
shows that the sequence~$(a^{n^2+1})$ is
realizable, so by~\eqref{equationbasiccongruence}
the last sum in~\eqref{equationNotUnsympathetic}
must be non-negative and divisible by~$n$ for all~$n>2$.
This shows that~$(a_n)$ is a realizable
sequence, and hence~$h\in\mathscr{P}(X,T)$.
To see that~$\mathscr{P}(X,T)$ is not everything,
notice that if the map exchanging~$1$ and~$3$
lies in~$\mathscr{P}(X,T)$, then~$a^3\leqslant a$,
which is impossible.
\end{example}

\section{Questions}\label{sectionQuestions}

\begin{enumerate}[(a)]
\item The simple arguments showing that realizable
sequences can be added and multiplied may be seen
using disjoint unions and products of dynamical
systems. Is there a similar argument showing that
monomials preserve realizability?
For example,
from a
`natural' system~$(X,T)$ with~$a_n=\fix_{(X,T)}(n)$
for all~$n\geqslant1$
(a smooth map on a compact
manifold, say), is there a
simple construction of a
system~$(X^{(2)},T^{(2)})$ with the
property that
\[
\fix_{(X^{(2)},T^{(2)})}(n)=a_{n^2}
\]
for all~$n\geqslant1$?
Of course the proof above notionally `constructs' such a system
because it can be used to
extract a formula for how many orbits of
each length such a map must have,
but in a far from natural or geometric way.
\item There is no {\it{a priori}}
reason for any given~$\mathscr{P}(X,T)$ to
be a monoid under composition of functions,
though~$\mathscr{P}$ clearly is. For
cases with~$\mathscr{P}(X,T)\supsetneq\mathscr{P}$,
what combinatorial properties of~$(\fix_{(X,T)}(n))$
determine the property that~$\mathscr{P}(X,T)$ is
a monoid?
\item Is there a sequence of
systems~$\bigl((X_n,T_n)\bigr)_{n\geqslant1}$
with the property that
\[
\mathscr{P}(X_n,T_n)
\supsetneq
\mathscr{P}(X_{n+1},T_{n+1})
\]
for all~$n\geqslant1$?
\item Can a non-trivial permutation of~$\mathbb{N}$
lie in~$\mathscr{P}$?
\item Is there a map~$T\colon X\to X$
with the property that the only polynomials
in~$\mathscr{P}(X,T)$ are monomials?
\item Is there a map~$T\colon X\to X$ with the property
that~$\mathscr{P}(X,T)=\mathscr{P}$?
\end{enumerate}

\providecommand{\bysame}{\leavevmode\hbox to3em{\hrulefill}\thinspace}
\providecommand{\MR}{\relax\ifhmode\unskip\space\fi MR }
\providecommand{\MRhref}[2]{%
  \href{http://www.ams.org/mathscinet-getitem?mr=#1}{#2}
}
\providecommand{\href}[2]{#2}

\end{document}